\documentclass{amsart}

\usepackage{amsmath,amsthm,amssymb,enumitem,verbatim,stmaryrd,xcolor,microtype,graphicx}
\usepackage[top=3.4cm,bottom=3.4cm,left=3.2cm,right=3.2cm]{geometry}
\usepackage{tikz-cd}
\usepackage[colorlinks=true]{hyperref}

\theoremstyle{plain}
\newtheorem{thm}{Theorem}[section]
\newtheorem{cor}[thm]{Corollary}
\newtheorem{lem}[thm]{Lemma}
\newtheorem{prop}[thm]{Proposition}
\newtheorem{conj}[thm]{Conjecture}
\newtheorem*{thm*}{Theorem}
\newtheorem*{lem*}{Lemma}

\theoremstyle{definition}
\newtheorem{df}[thm]{Definition}

\newtheorem{ex}[thm]{Example}
\newtheorem*{df*}{Definition}
\newtheorem*{ex*}{Example}
\newtheorem*{rem*}{Remark}

\theoremstyle{remark}

\newcommand\T{{\mathbf T}}

\newcommand\Z{{\mathbb Z}}

\newcommand\cT{{\mathcal T}}

\title{Csikv\'{a}ri's poset and Tutte polynomial}

\author{Changxin Ding}
\email{cding66@gatech.edu}
\address{School of Mathematics, Georgia Institute of Technology \\ Atlanta, Georgia 30332-0160, USA}

\begin{document}

\begin{abstract}
Csikv\'{a}ri constructed a poset on trees to prove that several graph functions attain extreme values at the star and the path among the trees on a fixed number of vertices. Reiner and Smith proved that the Tutte polynomials $\T(1,y)$ of cones over trees, which are the graphs obtained by attaching a cone vertex to a tree, have the described extreme behavior. They further conjectured that the result can be strengthened in terms of Csikv\'{a}ri's poset. We solve this conjecture affirmatively.

%Reiner and Smith studied the sandpile groups for graphs obtained by attaching a cone vertex to a tree and proved that the Tutte polynomials $\T(1,y)$ of such graphs have the described extreme behavior. They further conjectured that 

\end{abstract}
\maketitle

\section{Introduction}\label{Intro}

This paper aims to prove a conjecture of Victor Reiner and Dorian Smith. We first introduce some background. 

Fix a positive integer $n\geq 3$ and consider the set $\cT_n$ of all trees on $n$ vertices. In $\cT_n$, two extreme elements are the tree $\text{Path}_n$ with exactly $2$ leaves and the tree $\text{Star}_n$ with exactly $n-1$ leaves. Many graph functions $F(G)$ attain extreme values at $\text{Star}_n$ and $\text{Path}_n$ among trees in $\cT_n$. For example, Lov\'{a}sz and Pelik\'{a}n \cite{LP} proved that the star has the largest spectral radius and the path has the smallest spectral radius among all trees in $\cT_n$ (in short, we say that the function $F(G)$ can be the spectral radius of $G$); Zhou and Gutman \cite{GZ} proved that $F(G)$ can be the coefficients of the characteristic polynomial of the Laplacian matrix of $G$; P\'{e}ter Csikv\'{a}ri \cite{Csikvari1} proved that $F(G)$ can the number of closed walks of a fixed length $l$ in $G$. For more examples, see \cite{Csikvari2}.

%P\'{e}ter Csikv\'{a}ri \cite{Csikvari1, Csikvari2} proved that several graph functions $F(G)$ attain extreme values at $\text{Star}_n$ and $\text{Path}_n$ among trees in $\cT_n$. For example, the function $F(G)$ can be the largest eigenvalue of the adjacency matrix of $G$, the number of closed walks of a fixed length $l$ in $G$, etc. 

Among these works, Csikv\'{a}ri's is of most interest to us. Csikv\'{a}ri's main tool is an operation on $\cT_n$ called the \emph{generalized tree shift}, which makes $\cT_n$ a partially ordered set. We call it the \emph{Csikv\'{a}ri poset}. One basic feature of the Csikv\'{a}ri poset is that the tree $\text{Star}_n$ is the unique maximal element and the tree $\text{Path}_n$ is the unique minimal element \cite[Corollary 2.5]{Csikvari1}. 

For the definition of Csikv\'{a}ri poset, see Section~\ref{Pre}. For readers' convenience, we quote the example of $\cT_7$ from \cite{RS} here; see Figure~\ref{Hasse}. For the example of $\cT_6$ , see \cite[Figure 2]{Csikvari2}.

The generalized tree shift indeed generalizes many transformations for trees found in the literature; see \cite[Section 10]{Csikvari2} for a detailed discussion. Csikv\'{a}ri's tree shift was inspired by a graph transformation defined by Kelmans in \cite{Kelmans}. In \cite{BT}, the tree shift is called the \emph{KC-transformation} in honor of Kelmans and Csikv\'{a}ri. Csikv\'{a}ri's tree shift was used to prove that several graph functions $F(G)$ have the property described in the second paragraph; see \cite{Csikvari1, BT,Csikvari2}. 

\begin{figure}[ht]\label{Hasse}
            \centering
            \includegraphics[scale=0.15]{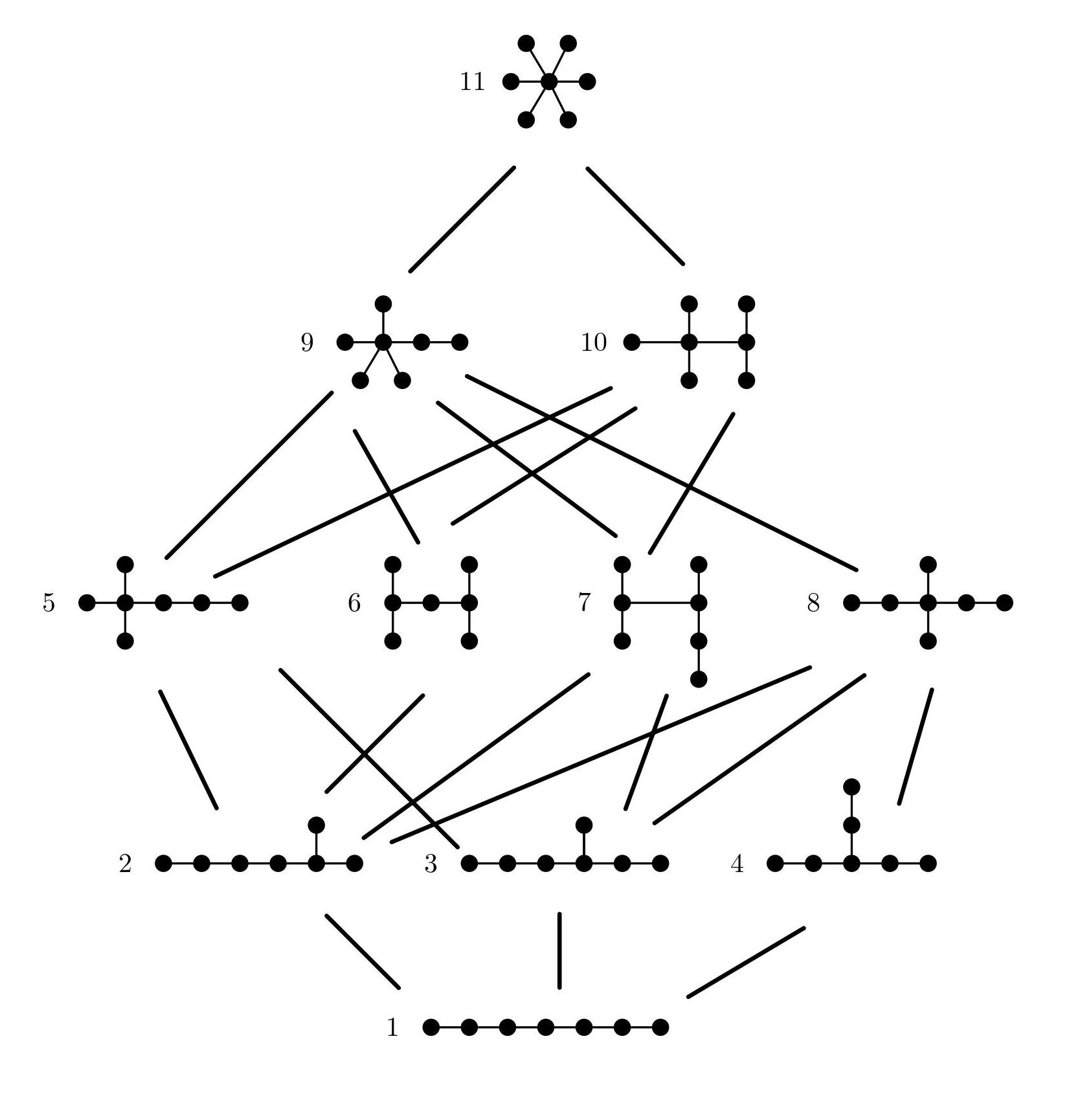}
        \caption{The Hasse diagram of the Csikv\'{a}ri poset on trees with $7$ vertices \cite{RS}.}
\end{figure}

In \cite{RS}, Victor Reiner and Dorian Smith studied the sandpile groups of graphs ${\text{Cone}(T)}$ obtained by attaching a cone vertex to a tree $T$. In particular, they studied the function
\[f(T):=\T_{\text{Cone}(T)}(1,y),\]
where $\T_G(x,y)\in\Z[x,y]$ denotes the Tutte polynomial of a graph $G$. 

In general, the coefficients of the Tutte polynomial $\T_G(1,y)$ encode some information from $G$ including external activities, the number of recurrent sandpile configurations (also known as critical configurations) at a given level, and the number of reduced divisors (also known as $G$-parking functions or superstable configurations) at a given level; see \cite{Merino2, Bernardi, BakerShokrieh}. Via duality, $\T_G(1,y)$ can be viewed as the $h$-vector of the cographic matroid $M^*(G)$ \cite{Merino3}.

Going back to Reiner and Smith's work, they proved the following result. 
\begin{thm}\label{old}\cite[Theorem 1.4]{RS} For a tree $T\in\cT_n$, the inequalities
\[f(\text{Star}_n)\leq f(T)\leq f(\text{Path}_n)\]
hold coefficientwise. 
\end{thm}

For readers' convenience, we quote the values $f(T)$ for all the trees $T\in\cT_7$ from \cite{RS}; see Table~\ref{table}. 
\begin{table}
\begin{center}
\begin{tabular}{ |c|c|c| } 
 \hline
 $T$&  $f(T)$ \\
  \hline\hline
 1 &  $y^6+7y^5+26y^4+63y^3+104y^2+122y+64$ \\
  \hline
  2 &  $y^6+7y^5+25y^4+59y^3+96y^2+104y+64$ \\
  \hline
 3 &  $y^6+7y^5+25y^4+58y^3+94y^2+104y+64$ \\
  \hline
 4 & $y^6+7y^5+25y^4+57y^3+92y^2+104y+64$ \\
  \hline
 5 &  $y^6+7y^5+24y^4+53y^3+83y^2+92y+64$ \\
  \hline
 6 &  $y^6+7y^5+24y^4+55y^3+89y^2+96y+64$ \\
  \hline
 7 &  $y^6+7y^5+24y^4+54y^3+86y^2+96y+64$ \\
  \hline
 8 &  $y^6+7y^5+24y^4+52y^3+80y^2+92y+64$ \\
  \hline
 9 &  $y^6+7y^5+23y^4+47y^3+68y^2+78y+64$ \\
  \hline
 10 &  $y^6+7y^5+23y^4+49y^3+76y^2+84y+64$ \\
  \hline
 11 &  $y^6+7y^5+22y^4+42y^3+57y^2+63y+64$ \\ 
 \hline
\end{tabular}
\caption{The values $f(T)$ for all the trees $T\in\cT_7$. See Figure~\ref{Hasse} for the labels of the trees.}
\end{center}
\label{table}
\end{table}
Victor Reiner and Dorian Smith conjectured and we will prove the following result.

\begin{conj}\label{conjecture}\cite[Conjecture 7.1]{RS} Let $T$ and $T'$ be two trees. 
If $T<T'$ in the Csikv\'{a}ri poset, then $f(T') \leq f(T)$ coefficientwise in $\Z[y]$.
\end{conj}

Note that the conjecture implies Theorem~\ref{old}. Since our proof does not rely on Theorem~\ref{old} (or its proof), we obtain a new proof of Theorem~\ref{old}. Comparing the proof in \cite{RS} and ours, one can notice that they deal with many inequalities while we mostly deal with equalities. 

Remarkably, our proof makes use of what Csikv\'{a}ri called ``General Lemma'' (\cite[Theorem 5.1]{Csikvari2}). In particular, when $T'$ covers $T$ in the Csikv\'{a}ri poset, we can factor $f(T)-f(T')$ into three polynomials with no negative coefficients (Proposition~\ref{mainresult}). Thus our result can be viewed as another successful application of Csikv\'{a}ri's method in \cite{Csikvari2}.

Our paper is arranged as follows. In Section~\ref{Pre}, we give the necessary definitions and notations. In Section~\ref{proof}, we present the proof of Conjecture~\ref{conjecture}.

\section{Preliminaries}\label{Pre}
%Fix a positive integer $n\geq 3$. Recall in Section~\ref{Intro} that $\cT_n$ denotes the set of all trees on $n$ vertices, $\text{Path}_n\in\cT_n$ denotes the tree with exactly $2$ leaves, and $\text{Star}_n\in\cT_n$ denotes the tree with exactly $n-1$ leaves. 

\subsection{Tutte polynomial and Cone of Tree}
By a graph, we mean a finite undirected graph possibly with parallel edges and loops. Recall that the Tutte polynomial $\T_G(x,y)$ of a graph $G$ can be defined recursively by the following contraction-deletion relation.

\begin{equation*}\label{Tutte-recursive-definition}
\T_G(x,y) = 
\begin{cases}
\T_{G\setminus e}(x,y)+ \T_{G/e}(x,y),
&\text{if  }e \text{ is neither a loop nor a bridge}, \\
y \cdot \T_{G \setminus e}(x,y),
&\text{if  e is a loop,}\\
x \cdot \T_{G/e} (x,y),
&\text{if  e is a bridge,}\\
1, &\text{if }G\text{ has no edges,}
\end{cases}
\end{equation*}
where $G\setminus e$ (resp. $G/e$) is the graph obtained from $G$ by deleting (resp. contracting) the edge $e$. It follows that the coefficients of Tutte polynomials are non-negative integers. 

By the contraction-deletion relation, it is easy to see that the trees in $\cT_n$ cannot be distinguished by their Tutte polynomials. Following \cite{RS}, we may consider the Tutte polynomials of the cones over trees to help with this issue.

\begin{df}
Let $G$ be a graph. \begin{enumerate}
    \item The \emph{cone} over the graph $G$ is the graph obtained from $G$ by adding an extra vertex $r$ and then connecting $r$ to each vertex of $G$, denoted by $\text{Cone}(G)$. 
    \item Define the function
\[f(G):=\T_{\text{Cone}(G)}(1,y)\in\Z[y].\]
\end{enumerate}  
\end{df}

%See \cite[Section 1.2]{RS} for more motivations to study cones over trees, which are related to the sandpile groups. The following theorem in their work is of our interest. 

\begin{df}
Throughout our paper, for two polynomials $f_1, f_2\in\Z[y]$, the inequality $f_1\leq f_2$ means that the coefficients of $f_2-f_1$ are all non-negative. 
\end{df}
In Section~\ref{Intro}, we have seen that the function $f(G)$ attains extreme values at $\text{Star}_n$ and $\text{Path}_n$ when taking values in $\cT_n$ (Theorem~\ref{old}). 

%\begin{thm}\label{old}\cite[Theorem 1.4]{RS} For a tree $G\in\cT_n$, 
%\[f(\text{Star}_n)\leq f(G)\leq f(\text{Path}_n),\]
%where $\text{Star}_n\in\cT_n$ is the tree with $n-1$ leaves and $\text{Path}_n\in\cT_n$ is the tree with $2$ leaves. 
%\end{thm}

%Csikv\'{a}ri \cite{Csikvari1, Csikvari2} proved that several graph functions $F(G)$ attain extreme values at $\text{Star}_n$ and $\text{Path}_n$ among trees in $\cT_n$. For example, $F(G)$ can be the largest eigenvalue of the adjacency matrix of $G$, the number of closed walks of a fixed length $l$ in $G$, etc. By Therorem~\ref{old}, the function $f(G)$ is another example. 

\subsection{Csikv\'{a}ri poset}
We first introduce a useful notation, which is also adopted in \cite{Csikvari2}. 

\begin{df}
For $i=1,2$, let $G_i$ be a graph with a distinguished vertex $v_i$. Let $G_1:G_2$ be the graph obtained from $G_1$ and $G_2$ by identifying the vertices $v_1$ and $v_2$.    
\end{df}
This operation depends on the vertices we choose, but we omit this in the notation for simplicity. It should be clear from the context what $v_1$ and $v_2$ are. Sometimes we also use the same label $v$ for the two distinguished vertices in $G_1$ and $G_2$. The operation is also known as a $1$-sum of $G_1$ and $G_2$ as a special case of clique-sums. 

The following operation on $\cT_n$ plays a central role in Csikv\'{a}ri's theory \cite{Csikvari1, Csikvari2} and our paper. 

\begin{df}\label{treeshift}
Let $P_k$ be a path in a tree $T$ with the vertex sequence $v_1,v_2,\cdots, v_k(k\geq 2)$ such that all the interior vertices $v_2, \cdots, v_{k-1}$ have degree two in $T$. (When $k=2$, there is no interior vertex.) By removing all the edges and interior vertices in the path $P_k$ from $T$, we obtain a tree $H_1$ with the distinguished vertex $v_1$ and a tree $H_2$ with the distinguished vertex $v_k$. Then the \emph{generalized tree shift} (with respect to $P_k$) transforms the tree $T$ into the tree $T':=(H_1:H_2):P_k$, where we identify the vertex $v_k$ in $H_1:H_2$ with the vertex $v_k$ in $P_k$. (See Figure~\ref{figureshift}.)
\end{df}
%\DCX{When $T'\neq T$, we denote $T\lessdot T'$. See Figure TBD for an illustration. Do I need this?} 

\begin{figure}[ht]
            \centering
            \includegraphics[scale=0.2]{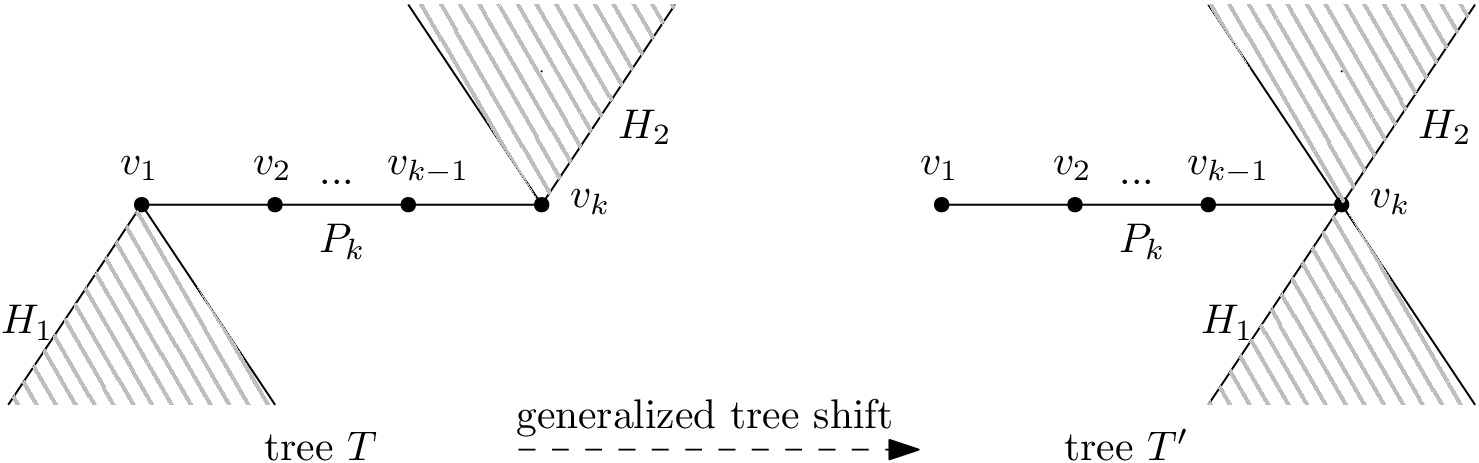}
        \caption{A generalized tree shift transforms a tree $T$ to anther tree $T'$ (Definition~\ref{treeshift}).}
            \label{figureshift}
\end{figure}

Observe that the generalized tree shift increases the number of leaves by $1$ unless $T$ and $T'$ are isomorphic graphs (when $v_1$ or $v_k$ is a leaf of $T$). Hence we may define the following poset. 
\begin{df}
For $T, T'\in\cT_n$, we denote $T\leq T'$ if $T'$ can be obtained from $T$ by some generalized tree shifts. This gives a partial order on $\cT_n$. We call it the \emph{Csikv\'{a}ri poset} on $\cT_n$. (See Figure~\ref{Hasse}.)
\end{df}

If a tree has two vertices that are not leaves, then one can increase the number of leaves by applying a generalized tree shift to the tree. This implies that the only maximal element of the Csikv\'{a}ri poset is the star. It is less trivial that the only minimal element is the path. 

\begin{prop}\label{maxmin}\cite[Corollary 2.5]{Csikvari1} In the Csikv\'{a}ri poset $(\cT_n, \leq)$, the tree $\text{Star}_n$ is the unique maximal element and the tree $\text{Path}_n$ is the unique minimal element. 
\end{prop}

%In \cite[Conjecture 7.1]{RS}, Victor Reiner and Dorian Smith conjectured the following result, which we will prove in the next section. 
%\begin{thm}\label{main}
%If $T<T'$ in Csikv\'{a}ri's poset, then $f(T')\leq f(T)$. 
%\end{thm}
%Note that Theorem~\ref{main} implies Theorem~\ref{old} by Proposition~\ref{maxmin}. Since our proof does not rely on Theorem~\ref{old} (or its proof), we obtain a new proof of Theorem~\ref{old}. Remarkably, our proof relies heavily on what Csikv\'{a}ri called ``General Lemma'' \DCX{see lemma or \cite[Theorem 5.1]{Csikvari2}}. Thus our result can viewed as another successful application of Csikv\'{a}ri's theory in \cite{Csikvari2}. \DCX{maybe describe their proof}

\section{Proof of Conjecture~\ref{conjecture}}\label{proof}

When the tree $T'$ covers $T$ in the Csikv\'{a}ri poset, we will prove the inequality $f(T')\leq f(T)$ by factoring $f(T)-f(T')$ into three polynomials with no negative coefficients (Proposition~\ref{mainresult}). To state this result, we need to define the following function. 
\begin{df} 
For a graph $G$ and a vertex $v$ of $G$, let $\varepsilon(v)$ be the edge of $\text{Cone}(G)$ connecting the vertex $v$ and the apex $r$ of the cone. Note that $\varepsilon(v)$ is a bridge if and only if $v$ is not adjacent to any other vertex of $G$. Define the function
\begin{equation*}
g_v(G):=
\begin{cases}
\T_{\text{Cone}(G)\backslash \varepsilon(v)}(1,y),
&\text{if }v \text{ is adjacent to some other vertex of }G, \\
0,
&\text{otherwise}.
\end{cases}    
\end{equation*}
\end{df}
\begin{prop}\label{mainresult}
Using the notations of Definition~\ref{treeshift}, we have for the trees $T$ and $T'$,\[f(T)-f(T')=y\cdot g_{v_1}(H_1)\cdot g_{v_k}(H_2)\cdot g_{v_1}(P_k).\] Consequently, $f(T')\leq f(T)$.
\end{prop}

\begin{cor}
Conjecture~\ref{conjecture} holds.     
\end{cor}
\begin{proof}
This is a direct consequence of Proposition~\ref{mainresult} and the definition of Csikv\'{a}ri poset. 
\end{proof}

Before we present the technical proof of Proposition~\ref{mainresult}, we introduce a notation and give an example first. 

In this section, any path on $k$ vertices will be denoted by $P_k$ where the vertices are labeled by $v_1,\cdots, v_k$ in order. Sometimes a formula or a sentence will involve two different paths where we use the same label for two different vertices; e.g. the paths $P_2$ and $P_3$ both have a vertex labeled by $v_1$. This abuse of notation should not cause any confusion in the context. 

\begin{ex}\label{example g}
This example demonstrates Proposition~\ref{mainresult}. By direct computations, we have \[g_{v_1}(P_2)=\T_{P_2}(1,y)=1,\]
\[g_{v_2}(P_3)=\T_{4\text{-cycle}}(1,y)=y+3,\text{ and}\]
\[g_{v_1}(P_3)=\T_{3\text{-cycle}}(1,y)=y+2.\]
Then we consider the data in Figure~\ref{Hasse} and Table~\ref{table}. The tree $T_2$ consists of a subtree $H_1=P_3$ with the distinguished vertex $v_1$, a subtree $H_2=P_3$ with the distinguished vertex $v_2$, and the path $P_3$ connecting the two distinguished vertices. The generalized tree shift transforms the tree $T_2$ into $T_8$. Then one can check that $f(T_2)-f(T_8)=y^4+7y^3+16y^2+12y=yg_{v_1}(P_3)g_{v_2}(P_3)g_{v_1}(P_3)$.
\end{ex}

Our main tool to prove Proposition~\ref{mainresult} is what Csikv\'{a}ri called ``General Lemma''. 
\begin{lem}\label{general lemma}(General Lemma)\footnote{By the proof of the General Lemma in \cite{Csikvari2}, if we only assume that the condition (1) holds for any two trees $G_1$ and $G_2$, then the conclusion of the lemma still holds.}\cite[Theorem 5.1]{Csikvari2} Let $f(G)$ be a graph polynomial in $y$. Assume there exists a graph polynomial $g_v(G)$ in $y$ whose inputs are a graph $G$ and a vertex $v$ of $G$ such that the following two conditions hold.
\begin{enumerate}
%\item For any two graphs $G_1$ with a distinguish vertex $v_1$ and $G_2$ with a distinguish vertex $v_2$, we have
%\[
%f(G_1:G_2)=c_1f(G_1)f(G_2)+c_2f(G_1)g_{v_2}(G_2)+c_2f(G_2)g_{v_1}(G_1)+c_3g_{v_1}(G_1)g_{v_2}(G_2),
%\]
%where $c_1,c_2,c_3$ are rational functions of $y$.
\item For any two graphs $G_1$ and $G_2$, we have
\[
f(G_1:G_2)=c_1f(G_1)f(G_2)+c_2f(G_1)g_{v}(G_2)+c_2f(G_2)g_{v}(G_1)+c_3g_{v}(G_1)g_{v}(G_2),
\]
where $c_1,c_2,c_3$ are rational functions of $y$ and $v$ is the identified vertex in $G_1:G_2$.

\item Denote \[q_v(G):=c_2f(G)+c_3g_{v}(G).\] We have that $q_{v_1}(P_2)$ is not a zero polynomial.

%Let $P_2$ and $P_3$ be the paths on two and three vertices, respectively. 
\end{enumerate}

Then the conclusion is 
\[f(T)-f(T')=\frac{g_{v_1}(P_3)-g_{v_2}(P_3)}{(q_{v_1}(P_2))^2}q_{v_1}(P_k)q_{v_1}(H_1)q_{v_k}(H_2),\]where the trees $T$ and $T'$ are as in Definition~\ref{treeshift}.
\end{lem}

Although in our paper the function $f$ has been defined, the function $f$  in the lemma could be any graph polynomial. Csikv\'{a}ri called it the General Lemma because it was used to prove ``$T<T'\Rightarrow f(T)\leq f(T')$'' for several graph polynomials $f$ in \cite{Csikvari2}. One difficulty of applying the General Lemma is to find the function $g_v$, which we have given for our function $f$.

%The main difficulty of applying the lemma is to find the function $g_v$.

%For a graph $G$ and a vertex $v$ of $G$, let $\varepsilon(v)$ be the edge of $\text{Cone}(G)$ connecting the vertex $v$ and the apex $r$ of the cone. Define the function
%\[g_v(G):=T_{\text{Cone}(G)\backslash \varepsilon(v)}(1,y),\]
%where if $v$ is an isolated vertex of $G$ (and hence $\varepsilon(v)$ is a bridge), then we use the convention that $g_v(G):=0$ to make the forthcoming Equation~(\ref{eq1}) hold even when $v$ is an isolated vertex of $G$.  

We will prove Proposition~\ref{mainresult} by checking that our functions $f$ and $g_v$ satisfy the conditions in the General Lemma. We first introduce an auxiliary function 
\[h_v(G):=\T_{\text{Cone}(G)/ \varepsilon(v)}(1,y).\]
By the contraction-deletion relation of Tutte polynomials, we get
\begin{equation}\label{eq1}
f(G)=g_v(G)+h_v(G). 
\end{equation}
It is not hard to show that $\T_{G_1:G_2}(x,y)=\T_{G_1}(x,y)\T_{G_2}(x,y)$. Hence we have
\begin{equation}\label{eq2}
h_v(G_1:G_2)=h_v(G_1)h_v(G_2),
\end{equation}
where $v$ is the identified vertex in $G_1:G_2$.

\begin{lem}\label{keyformula}
Let $v$ be the identified vertex in $G_1:G_2$. Then
\[f(G_1:G_2)=f(G_1)f(G_2)-yg_v(G_1)g_v(G_2).\]
\end{lem}
\begin{proof}
Observe that if the vertex $v$ is not adjacent to any other vertex of $G_1$, then we have $g_v(G_1)=0$ and $f(G_1:G_2)=f(G_1)f(G_2)$. Thus the desired equality holds. 

Now we use induction on the number edges in $G_1$ to prove \[f(G_1)f(G_2)-f(G_1:G_2)=yg_v(G_1)g_v(G_2).\]
The base case is that $G_1$ has no edge, which is covered by the observation above. For the inductive step, we may assume that there exists an edge $e$ of $G_1$ connecting $v$ and another vertex of $G_1$ (otherwise the case is covered by the observation). Let $\varepsilon(v)$ be the edge of $\text{Cone}(G_1)$ connecting the vertex $v$ and the apex of the cone.

Then by applying the contraction-deletion relation to the edge $e$ of $\text{Cone}(G_1)$, we get \[f(G_1)=f(G_1\setminus e)+\T_{\text{Cone}(G_1)/e}(1,y).\]
By further applying the contraction-deletion relation to the edge $\varepsilon(v)$ of $\text{Cone}(G_1)/e$ (which is one of the multiple edges produced by the contraction), we get
\begin{equation}\label{eq3}
f(G_1)=f(G_1\setminus e)+f(G_1/e)+yh_v(G_1/e). 
\end{equation}

Similarly, we apply the contraction-deletion relation to the edge $e$ in $\text{Cone}(G_1)\backslash \varepsilon(v)$. Note that $e$ is a bridge of $\text{Cone}(G_1)\backslash \varepsilon(v)$ if and only if $v$ is not adjacent to any other vertex in the graph $G_1\setminus e$, and the latter condition implies $g_v(G_1\setminus e)=0$. Thus we have\[g_v(G_1) = f(G_1/e)+g_v(G_1\setminus e),\]whether $e$ is a bridge of $\text{Cone}(G_1)\backslash \varepsilon(v)$ or not. Then by Equation~(\ref{eq1}), we get
\begin{equation}\label{eq4}
g_v(G_1)  = g_v(G_1/e)+h_v(G_1/e)+g_v(G_1\setminus e).
\end{equation}

The following computation finishes the proof, where IH means the induction hypothesis.
\begin{equation*}
\begin{split}
  & f(G_1)f(G_2)-f(G_1:G_2)   \\
\overset{(\ref{eq3})}{=}\ & (f(G_1\setminus e)+f(G_1/e)+yh_v(G_1/e))\cdot f(G_2)-f(G_1\setminus e:G_2)-f(G_1/e:G_2)-yh_v(G_1/e:G_2)\\
\overset{\text{IH}}{=}\ & yg_v(G_1\setminus e)g_v(G_2)+yg_v(G_1/e)g_v(G_2)+yh_v(G_1/e)f(G_2)-yh_v(G_1/e:G_2)\\
\overset{(\ref{eq2})}{=}\ & yg_v(G_1\setminus e)g_v(G_2)+yg_v(G_1/e)g_v(G_2)+yh_v(G_1/e)f(G_2)-yh_v(G_1/e)h_v(G_2)\\
\overset{(\ref{eq1})}{=}\ & yg_v(G_1\setminus e)g_v(G_2)+yg_v(G_1/e)g_v(G_2)+yh_v(G_1/e)g_v(G_2)\\
\overset{(\ref{eq4})}{=}\ & yg_v(G_1)g_v(G_2).
\end{split}
\end{equation*}
\end{proof}

\begin{proof}[Proof of Proposition~\ref{mainresult}]
Consider applying the General Lemma with $c_1=1, c_2=0, c_3=-y$. 

By Lemma~\ref{keyformula}, the first condition in the General Lemma holds. For the second condition, we have \[q_v(G)=-yg_{v}(G),\]and hence $q_{v_1}(P_2)=-yg_{v_1}(P_2)=-y\neq 0$ by Example~\ref{example g}. Thus the General Lemma can be applied in our setting, and a direct computation gives the desired equality. The inequality follows from the fact that the Tutte polynomials and hence $g_v(G)$ do not have negative coefficients. 
\end{proof}

%We conclude our paper with the following remark.
%\begin{rem}[On Theorem~\ref{old}]
%As mentioned in Section~\ref{Intro}, Theorem~\ref{old} is a direct consequence of Proposition~\ref{maxmin} and Conjecture~\ref{conjecture}. Hence we obtain a new proof of Theorem~\ref{old}. Comparing the two proofs in \cite{RS} and ours, one can notice that they dealt with many inequalities while we mostly deal with equalities. 
%\end{rem}

%\section{Concluding Remarks}

%\begin{rem}[On the graph Laplacian]
%Let $G$ be a graph on $n$ vertices. Let $L(G)$ be the \emph{Laplacian matrix} of $G$ (i.e., the adjacency matrix minus the diagonal matrix of vertex degrees). Then the \emph{Laplacian characteristic polynomial} of $G$ is defined as \[\psi_G(\lambda):=\det(\lambda I_n-L(G)).\]We write the Laplacian characteristic polynomial in the coefficient form as \[\psi_G(\lambda)=\sum_{k=0}^n (-1)^kc_k(G)\lambda^{n-k}.\]\begin{thm}\cite[Theorem 7]{GZ}
%For any tree $T\in\cT_n$ and any integer $k\in [1,n]$, the inequalities\[c_k(\text{Star}_n)\leq c_k(T)\leq c_k(\text{Path}_n)\]hold. 
%\end{thm}
%Now we claim that \[(-1)^n\psi_T(-1)=\T_{\text{Cone}(T)}(1,1).\]
%\end{rem}

\section*{Acknowledgements}
Thanks to Victor Reiner, Dorian Smith, and P\'{e}ter Csikv\'{a}ri for helpful discussions. Thanks to the referees for useful comments. 

\bibliography{Csikvari}

\begin{thebibliography}{10}

\bibitem{BakerShokrieh}
Matthew Baker and Farbod Shokrieh.
\newblock Chip-firing games, potential theory on graphs, and spanning trees.
\newblock {\em J. Combin. Theory Ser. A}, 120(1):164--182, 2013.

\bibitem{Bernardi}
Olivier Bernardi.
\newblock Tutte polynomial, subgraphs, orientations and sandpile model: new connections via embeddings.
\newblock {\em Electron. J. Combin.}, 15(1):Research Paper 109, 53, 2008.

\bibitem{BT}
B\'{e}la Bollob\'{a}s and Mykhaylo Tyomkyn.
\newblock Walks and paths in trees.
\newblock {\em J. Graph Theory}, 70(1):54--66, 2012.

\bibitem{Csikvari1}
P\'{e}ter Csikv\'{a}ri.
\newblock On a poset of trees.
\newblock {\em Combinatorica}, 30(2):125--137, 2010.

\bibitem{Csikvari2}
P\'{e}ter Csikv\'{a}ri.
\newblock On a poset of trees {II}.
\newblock {\em J. Graph Theory}, 74(1):81--103, 2013.

\bibitem{Kelmans}
A.~K. Kelmans.
\newblock On graphs with randomly deleted edges.
\newblock {\em Acta Math. Acad. Sci. Hungar.}, 37(1-3):77--88, 1981.

\bibitem{LP}
L.~Lov\'{a}sz and J.~Pelik\'{a}n.
\newblock On the eigenvalues of trees.
\newblock {\em Period. Math. Hungar.}, 3:175--182, 1973.

\bibitem{Merino2}
Criel Merino.
\newblock {\em Matroids, the {T}utte polynomial and the chip firing game.}
\newblock PhD thesis, University of Oxford,, 1999.

\bibitem{Merino3}
Criel Merino.
\newblock The chip firing game and matroid complexes.
\newblock In {\em Discrete models: combinatorics, computation, and geometry ({P}aris, 2001)}, volume~AA of {\em Discrete Math. Theor. Comput. Sci. Proc.}, pages 245--255. Maison Inform. Math. Discr\`et. (MIMD), Paris, 2001.

\bibitem{RS}
Victor Reiner and Dorian Smith.
\newblock Sandpile groups for cones over trees.
\newblock {\em Res. Math. Sci.}, 11(4):Paper No. 59, 24, 2024.

\bibitem{GZ}
Bo~Zhou and Ivan Gutman.
\newblock A connection between ordinary and {L}aplacian spectra of bipartite graphs.
\newblock {\em Linear Multilinear Algebra}, 56(3):305--310, 2008.

\end{thebibliography}
\bibliographystyle{plain}

\end{document}